\documentclass[reqno]{amsart}

\usepackage{amsmath, amsthm, amscd, amsfonts, amssymb, graphicx, color}
\newcommand{\vertiii}[1]{{\left\vert\kern-0.25ex\left\vert\kern-0.25ex\left\vert #1 
    \right\vert\kern-0.25ex\right\vert\kern-0.25ex\right\vert}}

\usepackage[bookmarksnumbered, colorlinks, plainpages]{hyperref}

\usepackage{lipsum}
\usepackage{fancyhdr}

\usepackage{algorithm,algorithmic}

\usepackage{enumerate}

\newtheorem{theorem}{Theorem}[section]
\newtheorem{lemma}[theorem]{Lemma}
\newtheorem{proposition}[theorem]{Proposition}
\newtheorem{corollary}[theorem]{Corollary}
\theoremstyle{definition}

\theoremstyle{remark}
\newtheorem{remark}[theorem]{Remark}
\numberwithin{equation}{section}

\def \l{L}
\def \h{H}
\begin{document}
\title{On bounds of logarithmic mean and mean inequality chain}
\author{Shigeru Furuichi}
\address{Department of Information Science, College of Humanities and Sciences, Nihon University, Setagaya-ku, Tokyo, 156-8550, Japan;}
\email{furuichi.shigeru@nihon-u.ac.jp}

\author{Mehdi Eghbali Amlashi}
\address{Department of Mathematics, Mashhad Branch, Islamic Azad University, Mashhad, Iran;}
\email{amlashi@mshdiau.ac.ir}


\thanks{The author (S.F.) was partially supported by JSPS KAKENHI Grant Number 21K03341.}

\date{}



\begin{abstract}
An upper bound of the logarithmic mean is given by a convex
conbination of the arithmetic mean and the geometric mean. In addition, a lower bound  of the logarithmic mean is given by a geometric bridge of the arithmetic mean and the geometric mean. 
In this paper, we study the bounds of the logarithmic mean.  We give operator inequalities and norm inequalities for the fundamental inequalities on the logarithmic mean. We give monotonicity of the parameter for the unitarily invariant norm of the Heron mean, and give its optimality as the upper bound of the unitarily invariant norm of the logarithmic mean. We study the ordering of the unitarily invariant norms for the Heron mean, the Heinz mean, the binomial mean and the Lehmer mean. Finally, we give a new mean inequality chain as an application of the point--wise inequality.
\end{abstract}

\maketitle
\section{Introduction}\label{sec1}
For any positive real numbers $a$ and $b$, the arithmetic mean, the geometric mean and the harmonic  means are respectively defined by 
$$
A:=A(a,b)=\frac{a+b}{2},\quad G:=G(a,b)=\sqrt{ab}, \quad
H:=H(a,b)=\frac{2ab}{a+b},
$$
and the logarithmic mean is defied by $L:=L(a,b)=\dfrac{a-b}{\log a-\log b}$ with $L(a,a)=a$. Here we write the binomial mean (power mean) by $B_p:=B_p(a,b)=\left(\dfrac{a^p+b^p}{2}\right)^{1/p}$ for $p\in\mathbb{R}$ and $a,b>0$. (e.g., see \cite{B2007} for these means.)
In \cite{TPLin1974}, the inequality $L\le B_{1/3}$ with  the optimality of $p=\dfrac{1}{3}$ was shown. In \cite{FY2013}, the inequality
$B_{1/3} \le \dfrac{2}{3}G+\dfrac{1}{3}A$ was also shown. Thus the P\'olya inequality \cite{Zou} (see also \cite[Lemma 2]{B2006} as an elegant proof),  $L\le \dfrac{2}{3}G+\dfrac{1}{3}A$ was interpolated by the binomial mean $B_{1/3}$. As the comparison between $L$ and $B_p$, the inequality $G=B_0\le L$ with  the optimality of $p=0$  was shown in \cite{TPLin1974}. Precisely, the inequalities $B_0<L<B_{1/3}$ for any distinct $a,b>0$ were shown in \cite{TPLin1974}. As for the lower bound of $L$, it is known that we have $G^{2/3}A^{1/3}\le L$ in \cite{LS1983}. See \cite{Yang2013} for the further precise estimations on the logarithmic mean.

 In this paper, we give new bounds of $L$ applying the self-improvement method to the fundamental inequalities 
 \begin{equation}\label{fund_ineq}
 G\le G^{2/3}A^{1/3}\le L \le \dfrac{2}{3}G+\dfrac{1}{3}A\le A,\quad (a,b>0).
 \end{equation}
 This is same procedure used in \cite{FY2013}. In our paper, we mainly focus on \eqref{fund_ineq} since the bounds of $L$ are given by $A$ and $G$ simply, although there exsits some better bounds \cite{Yang2013,Zheng2000}.
\section{Bounds of logarithmic mean}\label{sec2}
As we stated in Section \ref{sec1}, the lower bound of $L$ has been given in \cite{LS1983} by studying the properties of two parameter Stolarsky mean. We here give its alternative proof by elementary calculations for the convenience to the readers. 
\begin{lemma}\label{le2.1}
For $x>0$, we have
\begin{equation}\label{eq01}
\sqrt{x}\le \left(\frac{x(x+1)}{2}\right)^{1/3}\le \frac{x-1}{\log x}\le \left(\frac{x^{1/3}+1}{2}\right)^3\le \frac{2}{3}\sqrt{x}+\frac{1}{3}\left(\frac{x+1}{2}\right).
\end{equation}
\end{lemma}
\begin{proof}
The first inequality is easily proven by $\left(\dfrac{x(x+1)}{2}\right)^{1/3} = \left(\sqrt{x}\right)^{2/3} \left(\dfrac{x+1}{2}\right)^{1/3}\ge \left(\sqrt{x}\right)^{2/3}\left(\sqrt{x}\right)^{1/3}=\sqrt{x}$ for $x>0$.

Since two functions $r(x):=\left(\dfrac{x(x+1)}{2}\right)^{1/3}$ and $l(x):=\dfrac{x-1}{\log x}$ satisfy $x\cdot r(1/x)=r(x)$ and $x\cdot l(1/x)=l(x)$, it is sufficient to prove the second inequality of \eqref{eq01} for $x \ge 1$. To this end, we set the function
$$
f(x):=2(x-1)^3-x(x+1)(\log x)^3,\quad (x\ge 1).
$$
Then we have
\begin{eqnarray*}
&& f'(x)=6(x-1)^2-3(x+1)(\log x)^2-(2x+1)(\log x)^3,\\
&& f''(x)=12(x-1)-6(1+1/x)\log x-3(3+1/x)(\log x)^2-2(\log x)^3,\\
&& f^{(3)}(x)=\frac{3}{x^2} g(x),\quad g(x):=4x^2-2x-2-6x\log x+(1-2x)(\log x)^2.
\end{eqnarray*}
We also have
\begin{eqnarray*}
&&g'(x)=8(x-1)-2(5-1/x)\log x-2(\log x)^2,\\
&&g''(x)=\frac{8x^2-10x+2-2(1+2x)\log x}{x^2},\\
&&g^{(3)}(x)=\frac{6(x-1)+4(x+1)\log x}{x^3} \ge 0 \quad (x \ge 1),
\end{eqnarray*}
which implies $g''(x)\geq g''(1)=0$ so that $g'(x)\ge g'(1)=0$. Thus we have $g(x)\ge g(1)=0$ which implies $f''(x)\ge f''(1)=0$ so that $f'(x)\ge f'(1)=0$. Thus we have $f(x)\ge f(1)=0$ for $x \ge 1$. Therefore we have the second inequality of \eqref{eq01} for $x>0$. 

Similarly we set the function 
$$
h(x):=3(x+1)^3\log x-8(x^3-1),\quad (x\ge 1).
$$
Then we have
\begin{eqnarray*}
&& h'(x)=3\left(3(x+1)^2\log x-7x^2+3x+3+1/x\right),\\
&& h''(x)=3\left(6(x+1)\log x-11x+9+3/x-1/x^2\right)\\
&& h^{(3)}(x)=3\left(6\log x-5+6/x-3/x^2+2/x^3\right)\\
&& h^{(4)}(x)=\frac{18(x-1)(x^2+1)}{x^4} \ge 0,\quad (x \ge 1),
\end{eqnarray*}
which implies $h^{(3)}(x)\ge h^{(3)}(1) =0$ so that $h''(x) \ge h''(1)=0$. Therefore we have $h'(x)\ge h'(1)=0$ which implies $h(x)\ge h(1)=0$. Thus we have the third inequality of \eqref{eq01} for $x \ge 1$. Replacing $x$ by $1/x$ in the third inequality of \eqref{eq01}, we obtain the third inequality of \eqref{eq01} for $x >0$. The last inequality was shown in \cite[Lemma 1.1]{FY2013} by an  elementary calculation.
\end{proof}

From Lemma \ref{le2.1}, we obtain a lower bound of the logarithmic mean by $A$ and $G$ in the following.
\begin{corollary}\label{co2.2}
We have 
\[
G\le G^{2/3} A^{1/3}\leq L\le B_{1/3}\le \frac{2}{3}G+\frac{1}{3}A.
\]
\end{corollary}
\begin{proof}
Putting $x=\dfrac{a}{b}$ in \eqref{eq01}, we have the result. 
\end{proof}
\begin{remark}
If $\dfrac{2}{3}\le t \le 1$ and $0\le s \le \dfrac{2}{3}$, then we have trivially
\begin{equation}\label{bounds_LM_ineq_st}
 G^{t}A^{1-t}\le L \leq s G+(1-s) A.
\end{equation}
\end{remark}

The right hand side of \eqref{bounds_LM_ineq_st} for $s\in [0,1]$ is often called the Heron mean. A mean $M$ is often called a geometric bridge \cite{HKPR} of two means $M_1$ and $M_2$ if it is written by $M=M_1^{r}M_2^{1-r}$ for $r\in [0,1]$. So a lower bound $G^{2/3}A^{1/3}$ is a geometric bridge of $G$ and $A$ with $r=2/3$.
In the following proposition, we give the optimality of $t=\dfrac{2}{3}$ and $s=\dfrac{2}{3}$ in \eqref{bounds_LM_ineq_st}. 
\begin{proposition}\label{prop_2.4}
We have
\begin{equation}\label{prop_eq01}
\inf \left\{t\in[0,1]: \left(\sqrt{x}\right)^t\left(\frac{x+1}{2} \right)^{1-t}\le \frac{x-1}{\log x},\quad (x>0)\right\}=\frac{2}{3}
\end{equation}
and
\begin{equation}\label{prop_eq02}
\sup \left\{s\in[0,1]:  \frac{x-1}{\log x}\le s\sqrt{x}+(1-s)\left(\frac{x+1}{2}\right),\quad (x>0)\right\}=\frac{2}{3}.
\end{equation}
\end{proposition}

\begin{proof}
The condition in \eqref{prop_eq01} is equivalent to the inequality:
$$
t \ge {{\log \left( {\dfrac{{\dfrac{{x - 1}}{{\log x}}}}{{\dfrac{{x + 1}}{2}}}} \right)} \mathord{\left/
 {\vphantom {{\log \left( {\dfrac{{\dfrac{{x - 1}}{{\log x}}}}{{\dfrac{{x + 1}}{2}}}} \right)} {\log \left( {\dfrac{{\sqrt x }}{{\dfrac{{x + 1}}{2}}}} \right)}}} \right.
 \kern-\nulldelimiterspace} {\log \left( {\dfrac{{\sqrt x }}{{\dfrac{{x + 1}}{2}}}} \right)}}=:\log_{ratio}(x).
$$
By the first inequality in \eqref{fund_ineq}, we have
$$
0< \log_{ratio}(x) \le \dfrac{2}{3}.
$$
By repeating the l'H\^{o}pital's rule, we have
\begin{eqnarray*}
&& \lim_{x\to 1}\log_{ratio}(x)
 =\lim_{x\to 1}\frac{2\left(x^2-1-2x\log x\right)}{(x-1)^2\log x} =\lim_{x\to 1}\frac{4x\left(x-1-\log x \right)}{(x-1)\left(x-1+2x\log x\right)}\\
 && =\lim_{x\to 1}\frac{8\left(x-1\right)-4\log x}{4(x-1)+2(2x-1)\log x}=\lim_{x\to 1}\frac{8-4/x}{8-2/x+4\log x}=\frac{2}{3}.
\end{eqnarray*}
Thus we have \eqref{prop_eq01}. 

Similarly the condition in \eqref{prop_eq02} is equivalent to the inequality:
$$
s \le {{\left( {\frac{{x + 1}}{2} - \frac{{x - 1}}{{\log x}}} \right)} \mathord{\left/
 {\vphantom {{\left( {\frac{{x + 1}}{2} - \frac{{x - 1}}{{\log x}}} \right)} {\left( {\frac{{x + 1}}{2} - \sqrt x } \right)}}} \right.
 \kern-\nulldelimiterspace} {\left( {\frac{{x + 1}}{2} - \sqrt x } \right)}}=:{\rm diff}_{ratio}(x).
$$
By the second inequality in \eqref{fund_ineq} and $G\le L$, we have
$$
\frac{2}{3}\le {\rm diff}_{ratio}(x) \le 1.
$$
By repeating the l'H\^{o}pital's rule, we also have
\begin{eqnarray*}
&& \lim_{x\to 1} {\rm diff}_{ratio}(x)
 =\lim_{x\to 1}\frac{2(x-1)-2x\log x+x(\log x)^2}{\sqrt{x}(\sqrt{x}-1)(\log x)^2}\\
 && =\lim_{x\to 1}\frac{2\sqrt{x}\log x}{4(\sqrt{x}-1)+(2\sqrt{x}-1)\log x}
 =\lim_{x\to 1}\frac{\sqrt{x}(2+\log x)}{4\sqrt{x}-1+\sqrt{x}\log x}=\frac{2}{3}.
\end{eqnarray*}
Thus we have \eqref{prop_eq02}.
\end{proof}

Applying same approach given in \cite{FY2013} to the fundamental inequalities \eqref{fund_ineq}, we give the better bound of $L$ in the following.
 \begin{theorem}\label{theorem_sum_ref}
 For $a,b>0$ and $m\in \mathbb{N}$, we have
 \begin{eqnarray}
 &&\frac{1}{m} \sum_{k=1}^ma^{\frac{2k-1}{2m}}b^{\frac{2m-(2k-1)}{2m}} \le \left(\frac{1}{m}\sum_{k=1}^ma^{\frac{m-k}{m}}b^{\frac{k-1}{m}}\right)\left\{\frac{a^{\frac{1}{m}}\left(a^{\frac{1}{m}}+b^{\frac{1}{m}}\right) b^{\frac{1}{m}}}{2}\right\}^{1/3} \nonumber \\
 && \le L \le \frac{2}{3m}\sum_{k=1}^ma^{\frac{2k-1}{2m}}b^{\frac{2m-(2k-1)}{2m}}+\frac{1}{3m}\left(\sum_{k=0}^ma^{\frac{k}{m}}b^{\frac{m-k}{m}}-\frac{a+b}{2}\right) \nonumber \\
 && \le \frac{1}{m}\left(\sum_{k=0}^ma^{\frac{k}{m}}b^{\frac{m-k}{m}}-\frac{a+b}{2}\right).\label{sum_ref_ineq01}
 \end{eqnarray}
 \end{theorem}
 \begin{proof}
 Inserting $x:=t^{1/m}$ in the fundamental inequalities 
 $$
 \sqrt{x}\le \left(\frac{x(x+1)}{2}\right)^{1/3}\le \frac{x-1}{\log x} \le \frac{2}{3}\sqrt{x}+\frac{1}{3}\left(\frac{x+1}{2}\right)\le \frac{x+1}{2},\quad (x>0)
 $$
 which is due to \eqref{fund_ineq}, and multipying $t^{\frac{m-1}{m}}+t^{\frac{m-2}{m}}+\cdots +t^{\frac{1}{m}}+1$ to both sides, and then putting $t:=a/b>0$ in the obtained inequalities and multiplying $b$ to both sides, we have the desired results.
 \end{proof}
  It is natural that Theorem \ref{theorem_sum_ref} improves \cite[Lemma 2.5, Lemma 3.4]{FY2013} since we used the tight fundamental inequalities \eqref{fund_ineq}, compared with the inequalities $G\le L \le A$, which were used in \cite{FY2013} as the fundamental inequalities. 
 \begin{corollary}
 We have the inequalities
   \begin{eqnarray*}
 && G^{1/2}\left(\frac{A+G}{2}\right)^{1/2}\le G^{1/3}\left(\frac{A+G}{2}\right)^{2/3}\le L \\
 &&\le \frac{2}{3} G^{1/2}\left(\frac{A+G}{2}\right)^{1/2}+\frac{1}{3}\left(\frac{A+G}{2}\right)\le \frac{A+G}{2}.
  \end{eqnarray*}
  \end{corollary}
  \begin{proof}
If we take $m=2$ in \eqref{sum_ref_ineq01}, then we have
  \begin{eqnarray*}
&& a^{1/4}b^{1/4}\left(\frac{\sqrt{a}+\sqrt{b}}{2}\right)\le \left(\sqrt{ab}\right)^{1/3}\left(\frac{\sqrt{a}+\sqrt{b}}{2}\right)^{4/3}\le L\\
&&  \le \frac{2}{3}a^{1/4}b^{1/4}\left(\frac{\sqrt{a}+\sqrt{b}}{2}\right)+\frac{1}{6}\left(\frac{a+b}{2}+\sqrt{ab}\right)\le \frac{1}{2}\left(\frac{a+b}{2}+\sqrt{ab}\right),
  \end{eqnarray*}
 which implies the result.
 \end{proof}
   Taking a large $m\in \mathbb{N}$, we can give tighter bounds of $L$, although it may be complicated expressions, however.

Based on \eqref{bounds_LM_ineq_st}, we obtain matrix inequalities for  positive definite matrices.
For  positive definite matrix $S$ and  positive semi-definite matrix $T$, the weighted arithmetic mean and the weighted geometric mean are defined by
$$
S\nabla_v T:=(1-v)S+vT,\quad S\sharp_vT:=S^{1/2}\left(S^{-1/2}TS^{-1/2}\right)^vS^{1/2},\quad (0\le v \le 1).
$$
Cosidering $S_{\varepsilon}:=S+\varepsilon I$ and the limit argument $\varepsilon \to 0$, $S\sharp_vT$ is defined for positive semi-definite matrices $S$ and $T$. For simplidity, we use $\nabla$ and $\sharp$ instead of $\nabla_{1/2}$ and $\sharp_{1/2}$, respectively.
The logarithmic mean for positive semi-definite matrices  $S$ and $T$ is defined by 
$$
S\ell T:=\int_0^1S\sharp_vTdv.
$$

Then we have the following results.
\begin{theorem}
Let $\dfrac{2}{3}\le t \le 1$ and $0\le s \le \dfrac{2}{3}$. For a positive definite matrix  $S$ and a positive semi-definite matrix $T$, we have
\begin{equation}\label{theorem_op_ineq01}
S\sharp T\le \left(S\sharp_{t/2}T\right)S^{-1}\left(S\sharp_{1-t}\left(S\nabla T\right)\right)\le S\ell T\le s S\sharp T+(1-s)S\nabla T \le S\nabla T.
\end{equation}
We also have for $p\ge \dfrac{1}{3}$,
\begin{equation}\label{theorem_op_ineq02}
S\ell T\le S^{1/2}\left(\dfrac{\left(S^{-1/2}TS^{-1/2}\right)^p+I}{2}\right)^{1/p}S^{1/2},
\end{equation}
where the right hand side of \eqref{theorem_op_ineq02} is called the matrix power mean \cite{DDF}.
\end{theorem}

\begin{proof}
Since $\displaystyle{\frac{x-1}{\log x}=\int_0^1x^vdv}$, we have
$$
\sqrt{x} \le (\sqrt{x})^t\left(\frac{x+1}{2}\right)^{1-t}\le \int_0^1x^vdv \le s\sqrt{x}+(1-s)\left(\frac{x+1}{2}\right)\le \frac{x+1}{2}
$$
from \eqref{bounds_LM_ineq_st}.
Inserting $x:=S^{-1/2}TS^{-1/2}$ in the above inequalities and multiplying $S^{1/2}$ to both sides,
we obtain \eqref{theorem_op_ineq01}.

Since we have $y \log y\ge (y+1)\log \left(\dfrac{y+1}{2}\right)$ for $y>0$, we have for $x>0$ and $p\in \mathbb{R}$,
$$
\frac{dB_p(x,1)}{dp}=\frac{1}{p^2(x^p+1)}\left(\frac{x^p+1}{2}\right)^{1/p}\left\{x^p\log x^p-(x^p+1)\log \left(\frac{x^p+1}{2}\right)\right\}\ge 0.
$$
This monotonicity on $p\in \mathbb{R}$ with the third inequality in \eqref{eq01} gives \eqref{theorem_op_ineq02} by Kubo-Ando theory \cite{KA}.
\end{proof}

It is natural that we do not have the ordering 
$$
\left(\dfrac{x^p+1}{2}\right)^{1/p}\le s \sqrt{x}+(1-s)\left(\dfrac{x+1}{2}\right)
$$
under the assumption as $p\ge \dfrac{1}{3}$ and $0\le s \le \dfrac{2}{3}$. (Take $p=1$ and $s=\dfrac{2}{3}$.)
As for the comparison  between the binomial mean and the Heron mean, see \cite[Lemma 2.1]{DDF}.
We give the comparison of these means in the end of Section \ref{sec3}.

\section{Norm inequalities}\label{sec3}

 We easily find that the logarithmic mean $L$ and 
 $$K_r:=K_r(a,b)=\left(\sqrt{ab}\right)^r\left(\dfrac{a+b}{2}\right)^{1-r},\,\,(0\le r \le 2)$$ are symmetric homogeneous means, since we have for $x>0$
 $$
 \frac{dK_r(x,1)}{dx}=\left(\dfrac{2\sqrt{x}}{x+1}\right)^r\left(\dfrac{r+(2-r)x}{4x}\right)\ge 0,
 $$
 if $0\le r \le 2$. We easily find $K_0=A,\,\, K_1=G$ and $K_2=H$. Thus the mean $K_r$ interpolates the arithmetic mean, the geometric mean and the harmonic mean. We also find $K_r$ is decreasing in  $r\in\mathbb{R}$ by $\dfrac{dK_r}{dr}=G^r A^{1-r}\log\dfrac{G}{A}\le 0$. The mean $K_r$ is the geometric bridge of the means $G$ and $A$ if $0\le r \le 1$. It is known \cite[Proposition 1.2]{HK1999India}  (see also \cite[Proposition 2.5]{K2011}) that the point-wise ordering such as $K_{2/3}(a,b)\le L(a,b)$ for $a,b>0$ is equivalent to the Hilbert-Schmidt (Frobenius) norm inequality $\left\| K_{2/3}(S,T)X\right\|_2\le \left\| L(S,T)X\right\|_2$ for positive semi-definite matrices $S,T$ and arbitrary matrix $X$. 
 The  symbol $\left\| \cdot\right\|_2$ means the Hilbert-Schmidt (Frobenius) norm. 
 Here we denote the set of all $n\times n$ complex matrices by $M_n(\mathbb{C})$.
In order to give norm inequalities for means, we have to state the definition of a \lq\lq matrix mean'' $M(S,T)X$ for a given symmetric homogeneous mean $M(a,b)$ according to \cite[Eq.(1.1)]{HK1999India} (see also \cite[Chapter 3]{HK_book}). 
For positive semi-definite matrices $S$ and $T$, there exsits unitary matrices $U$ and $V$ such that
$S=U{\rm diag}(\lambda_1,\cdots,\lambda_n)U^*$ and $T=V{\rm diag}(\mu_1,\cdots,\mu_n)V^*$, respectively. Then we define a matrix mean $M(S,T)X$ for a given symmetric homogeneous mean $M(a,b)$, positive semi-definite matrices $S,T$ and any $X \in M_n(\mathbb{C})$ by
$$
M(S,T)X:=U\left(\left[M(\lambda_i,\mu_j)\right]_{i,j=1,\cdots,n}\circ \left(U^*XV\right)\right)V^*,
$$
where the symbol $\circ$ represents the Hadamard product. It is also written by
 $$
 M(S,T)X=\sum_{i=1}^n\sum_{j=1}^nM(\lambda_i,\mu_j)P_iXQ_j,
 $$
 where $S=\sum\limits_{i=1}^n \lambda_i P_i$ and $T=\sum\limits_{j=1}^n\mu_jQ_j$ are spectral decompositions. In the sequel, $S$ and $T$ represent positive semi-definite matrices and $X$ also does arbitrary matrix in $M_n(\mathbb{C})$, unless otherwise specified. If it is written
 by $M(a,b)=\sum\limits_{k=1}^l f_k(a)g_k(b)$ for functions $f_k$ and $g_k$, then we easily find the natural expressions as 
 $$
 A(S,T)X=\dfrac{SX+XT}{2},\,\,L(S,T)X=\int_0^1 S^vXT^{1-v}dv,\,\, G(S,T)X=S^{1/2}XT^{1/2}.
 $$ 
However, we do not find any explicit expression for $K_{r}(S,T)X$ since $K_{r}(a,b)$ can not be written by $\sum\limits_{k=1}^l f_k(a)g_k(b)$ for some functions $f_k$ and $g_k$.
 
From the inequality \eqref{fund_ineq}, we have the following norm inequality:
 \begin{equation}
 \hspace*{-0.3cm}\left\| S^{1/2}XT^{1/2} \right\|_2\le  \left\|K_{2/3}(S,T)X\right\|_2  \le \left\| \int_0^1 S^vXT^{1-v}dv \right\|_2
 \le \left\| H_{2/3}(S,T)X \right\|_2, \label{HS_norm_ineq01}
 \end{equation}
where the third inequality of \eqref{HS_norm_ineq01} has already been shown in \cite[Theorem 2]{Zou}, and the Heron mean is defined by $H_s:=H_s(a,b)=sG(a,b)+(1-s)A(a,b),\,\,(0\le s \le 1)$.

Since the Hilbert-Schmit norm is one of the unitarily invariant norm, we are interested whether the inequalities \eqref{HS_norm_ineq01} will be generalized to the unitarily invariant norm inequalities.
For this purpose, we have only to prove the relation such as $K_{2/3} \preceq  L$, then we obtain the unitarily invariant norm inequality $\vertiii{ K_{2/3}(S,T)X } \le \vertiii{ L(S,T)X }$ for positive semi-definite matrices $S,T\in M_n(\mathbb{C})$ and arbitrary matrix $X\in M_n(\mathbb{C})$.  The symbol $\vertiii{\cdot}$ means the unitarily invariant norm. Here $K_{2/3} \preceq  L$ means that the function $K_{2/3}(e^t,1)/L(e^t,1)$ is positive definite  on $\mathbb{R}$. This is equivalent to the matrix
\[{\left[ {\frac{{K_{2/3}\left( {{t_i},{t_j}} \right)}}{{L\left( {{t_i},{t_j}} \right)}}} \right]_{i,j = 1,2, \cdots ,n}}\]
is positive semi-definite for any $t_1,t_2,\cdots,t_n>0$ and $n\in\mathbb{N}$. See \cite{B2007,HK1999India,K2011} for details. In this section, we often use the following statement:
$$
\vertiii{N_1(S,T)X}\le \vertiii{N_2(S,T)X}\Longleftrightarrow \,\,{\rm condition \,\,(C)}
$$
for two symmetric homogeneous means $N_1$ and $N_2$.
This notation means that the condition (C) is a necessary and sufficient condition of the inequality for all unitarily invariant norm and all positive semi-definite matrices $S,T\in M_n(\mathbb{C})$ and any matrix $X\in M_n(\mathbb{C})$ for all $n\in \mathbb{N}$. We often use the simple notation as equivalently
$$
N_1\preceq N_2 \Longleftrightarrow \,\,{\rm condition \,\,(C)}.
$$ 

In Section \ref{sec2}, we proved the point-wise inequality $K_{2/3} \le L$ and the inequalities $L\le B_{1/3}\le H_{2/3}$ are known as stated in Section \ref{sec1}. As it was shown in \cite[Example 3.5]{HK1999India}, $L\preceq B_{1/3}$ does not hold in general in spite of $L\le B_{1/3}$. 
In Remark \ref{remark_3.5} below, we state on the relation $L\preceq H_{2/3}$.
In this section, we mainly study the properties and the inequalities for the bounds $K_r$ and $H_s$ of the logarithmic mean given in \eqref{HS_norm_ineq01}. We also study the comparisons of some means.
We start from the following inequalities.

\begin{proposition}\label{sec3_prop3.1}
For $0\le r \le 1$ and $S,T\ge 0$, we have
 \begin{equation}\label{UI_norm_ineq01}
 \vertiii{S^{1/2}XT^{1/2}} \le \vertiii{K_r(S,T)X} \le \frac{1}{2}\vertiii{SX+XT} 
 \end{equation}
\end{proposition}

\begin{proof}
We compute
$$
\frac{G(e^t,1)}{K_r(e^t,1)}=\frac{e^{t/2}}{e^{rt/2}\left(\dfrac{e^t+1}{2}\right)^{1-r}}=\left(\frac{1}{\cosh t/2}\right)^{1-r}.
$$
 It is known that $\dfrac{1}{\cosh x}$ is infinitely divisible in  \cite[Lemma 2 (ii)]{K2014}. (See also \cite[Proposition 6]{BK2007}.)  A function $f$ is called  infinitely divisible if the function $\left(f(x)\right)^r$ for any $r\ge 0$ is positive definite \cite{BK2007}. Thus $\dfrac{G(e^t,1)}{K_r(e^t,1)}$ is positive definite for $0\le r \le 1$ so that we have the first inequality in \eqref{UI_norm_ineq01}.
 
 We also compute
$$
\frac{K_r(e^t,1)}{A(e^t,1)}=\frac{e^{t/2}\left(\cosh t/2\right)^{1-r}}{e^{t/2}\cosh t/2}=\left(\frac{1}{\cosh t/2}\right)^{r}.
$$
Thus  $\dfrac{K_r(e^t,1)}{A(e^t,1)}$ is positive definite for $0\le r \le 1$ so that we have the second inequality in \eqref{UI_norm_ineq01}.
 \end{proof}

The power difference mean (See \cite[Section 5]{HK_book}): 
 \[M_u:={M_u}\left( {a,b} \right) = \left\{ \begin{array}{l}
\dfrac{{u - 1}}{u} \times \dfrac{{{a^u} - {b^u}}}{{{a^{u - 1}} - {b^{u - 1}}}}\,\,\,\,\,\,\,\left( {a \ne b} \right)\\
\,\,\,\,\,\,a\,\,\,\,\,\,\,\,\,\,\,\,\,\,\,\,\,\,\,\,\,\,\,\,\,\,\,\,\,\,\,\,\,\,\,\,\,\,\,\,\,\,\,\,\,\,\,\,\,\,\,\,\,\left( {a = b} \right),
\end{array} \right.\]
interpolates $A,G,L$ and $H$. However, the mean $K_r$ interpolates $A,G$ and $H$ but does not do $L$. We have the relation $K_{r}\le L$ for $r \ge \dfrac{2}{3}$. We can not find $r\in(0,1)$ such that $L\le K_r$ since
we have 
\begin{eqnarray*}
\frac{L(x,1)}{K_r(x,1)}&=&\frac{x-1}{\log x}\times \frac{1}{x^{r/2}\left(\dfrac{x+1}{2}\right)^{1-r}}\ge \frac{1}{2}\times\frac{x}{\log x}\times \frac{1}{x^{r/2}x^{1-r}}\\
&=&\frac{\left(\sqrt{x}\right)^r}{\log \sqrt{x}}\longrightarrow \infty\,\,\,({\rm as}\,\,\,x\longrightarrow\infty).
\end{eqnarray*}
In the above inequality, we used the fact $x-1\ge \dfrac{x}{2}$ and $\dfrac{x+1}{2}\le x$ for a sufficiently large $x>0$.

In the following remarks, we compare the mean $K_r$ with the logarithmic mean $L$, the Heron mean $H_s$, the binomial mean $B_p$, the power difference mean $M_u$ and the Heinz mean: 
$$Hz_v:=Hz_v(a,b)=\dfrac{1}{2}\left(a^vb^{1-v}+a^{1-v}b^v\right),\,\,(0\le v \le 1).$$
\begin{remark}
\begin{itemize}
\item[(i)] Although we have the second inequality in \eqref{HS_norm_ineq01}, the inequality $\vertiii{K_{2/3}(S,T)X} \le \vertiii{L(S,T)X}$ does not hold. To show this, we compute
$$
\frac{K_{r}(e^{2t},1)}{L(e^{2t},1)}=\frac{t\left(\cosh t\right)^{1-r}}{\sinh t}=:\varphi_r(t)
$$
which is not positive definite for $r=2/3$  by the following example.
We calculate the eigenvalues
of the $3\times 3$ matrix ${\left[ {\varphi_{2/3}\left( {{t_i} - {t_j}} \right)} \right]_{i,j = 1,2,3}}$ with $t_1:=1,\,\,t_2:=2,\,\,t_3:=3$. This is same setting of \cite[Example 3.5]{HK1999India}. 
Three eigenvalues of the matrix 
$${\left[ {\varphi_{2/3}\left( {{t_i} - {t_j}} \right)} \right]_{i,j = 1,2,3}} 
\simeq \left( {\begin{array}{*{20}{c}}
{1}&{0.983295}&{0.857656}\\
{0.983295}&{1}&{0.983295}\\
{0.857656}&{0.983295}&{1}
\end{array}} \right)$$ 
are approximately obtained as $2.88404, 0.142344$ and $-0.026381$ by the numerical computations. 
Thus the function $\varphi_{2/3}(t)$ is not positive definite, that is $K_{2/3}\npreceq L$, which is the difference between the Hilbert-Schmidt norn inequality and the unitarily invariant norm inequality.
In addition,  one of the eigenvalues of the $5\times 5$ matrix  ${\left[ {\varphi_{0.9}\left( {{t_i} - {t_j}} \right)} \right]_{i,j = 1,\cdots,5}}$ with $t_i:=i$ for $i=1,\cdots,5$, takes the negative value by numerical computations.
Thus these numerical computations support the conjecture that the function $\varphi_r(t)$ for $r<1$, is not positive definite.

Since $\varphi_r(t)=\dfrac{t}{\sinh t}\times\left(\dfrac{1}{\cosh t}\right)^{r-1}$,
the function $\varphi_r(t)$ is positive definite if $1\le r \le 2$. Thus we have $\vertiii{K_r(S,T)}\le \vertiii{G(S,T)X} \le \vertiii{L(S,T)X}$ for $1\le r \le 2$, by $K_1=G$ and Proposition \ref{prop_3.4K} below.

\item[(ii)] We have
$$
\frac{K_r(e^{2t},1)}{H_s(e^{2t},1)}=\frac{1}{1-s}\times \frac{\cosh t}{\dfrac{s}{1-s}+\cosh t}\times \left(\frac{1}{\cosh t}\right)^r.
$$
Since $\dfrac{1}{\cosh t}$ is  infinitely divisible, the function  $\dfrac{K_r(e^t,1)}{H_s(e^t,1)}$ is positive definite if $\dfrac{s}{1-s}\le 0$ by \cite[Theorem 7.1]{K2011}. Since $0\le s \le 1$, we have  $K_r \preceq H_0=A$ for $ r \ge 0$. We also have
$$
1\le r \le 2 \,\,\,{\rm and}\,\,\, s\in[0,1/2]\cup \{1\} \,\,\,  \Longrightarrow K_r \preceq H_s,
$$
from 
$$
\frac{K_r(e^{2t},1)}{H_s(e^{2t},1)}=\frac{1}{1-s}\times \frac{1}{\dfrac{s}{1-s}+\cosh t}\times \left(\frac{1}{\cosh t}\right)^{r-1}
$$
and $\dfrac{1}{\beta+\cosh x}$ is positive definite if and only if $-1<\beta \le 1$ \cite[Eq.(7.1)]{K2011}.
\item[(iii)]
We have
$$ \frac{K_r(e^{2t},1)}{B_p(e^{2t},1)}=\frac{\left(\cosh t\right)^{1-r}}{\left(\cosh pt\right)^{1/p}}
=\frac{\cosh t}{\cosh pt}\left(\frac{1}{\cosh pt}\right)^{\frac1p-1}\left(\frac{1}{\cosh t}\right)^r.
$$
Since the function $\dfrac{\cosh ax}{\cosh bx}$ is positive definite if and only if $a \le b$,
the above function is positive definite if $p\ge 1$ and $\dfrac{1}{p}-1\ge 0$ and $r \geq 0$. Thus we have $K_r \preceq B_1=A$ for $ r \ge 0$. We also have $K_r \preceq B_p $ for $1\le r \le 2$ and $ p\ge 0$,
from
$$
\frac{K_r(e^{2t},1)}{B_p(e^{2t},1)}
=\left(\frac{1}{\cosh pt}\right)^{\frac1p}\left(\frac{1}{\cosh t}\right)^{r-1}\,\,\,{\rm with}\,\,\,
\lim_{p\to 0}\cosh^{1/p}(pt)=1.
$$
\item[(iv)] We have
\begin{eqnarray*}
&&\frac{K_r(e^{2t},1)}{M_u(e^{2t},1)}=\frac{u}{u-1}\times \left(\cosh t\right)^{1-r}\times \frac{\sinh (u-1)t}{\sinh ut}\\
&& =\frac{u}{2(u-1)}\times \left(\dfrac{1}{\cosh t}\right)^r\times \dfrac{\cosh t}{\cosh ut/2}\times \dfrac{\sinh (u-1)t}{\sinh ut/2},
\end{eqnarray*}
which is positive definite  if  $1\le \dfrac{u}{2}$ and $u-1\le \dfrac{u}{2}$ which is equivalent to $u=2$. Thus we have $K_r \preceq M_2=A$ for $0\le r \le 1$.
We also have $K_r \preceq M_u$ for $1\le r \le 2$ and $u\ge 1$
from
$$
\frac{K_r(e^{2t},1)}{M_u(e^{2t},1)}=\frac{u}{u-1}\times\left(\dfrac{1}{\cosh t}\right)^{r-1}\times \dfrac{\sinh (u-1)t}{\sinh ut}
$$
with $\lim\limits_{u\to 1}\dfrac{K_r(e^{2t},1)}{M_u(e^{2t},1)}=\left(\dfrac{1}{\cosh t}\right)^{r-1}\times \dfrac{t}{\sinh t}$.
\item[(v)] We have
\[\frac{{{K_r}\left( {{e^{2t}},1} \right)}}{{Hz_{v}\left( {{e^{2t}},1} \right)}} = \left\{ \begin{array}{l}
\frac{{\cosh t}}{{\cosh \left( {2v - 1} \right)t}} \times {\left( {\frac{1}{{\cosh t}}} \right)^r},\,\,\,\,\,\,\,\,\,\,\left( {\frac{1}{2} \le v \le 1} \right),\\
\frac{{\cosh t}}{{\cosh \left( {1 - 2v} \right)t}} \times {\left( {\frac{1}{{\cosh t}}} \right)^r},\,\,\,\,\,\,\,\,\,\,\left( {0 \le v \le \frac{1}{2}} \right).
\end{array} \right.\]
Since the function $\dfrac{\cosh ax}{\cosh bx}$ is positive definite if and only if $a \le b$, we have
$$
0\le r \le 1 \,\,\,{\rm and}\,\,\,v=\{0,1\} \Longrightarrow K_r \preceq Hz_{v} .
$$
We also have
$$
1\le r \le 2 \,\,\,{\rm and}\,\,\, 0\le v \le 1 \Longrightarrow K_r \preceq Hz_{v}
$$
from
\[\frac{{{K_r}\left( {{e^{2t}},1} \right)}}{{Hz_{v}\left( {{e^{2t}},1} \right)}} = \left\{ \begin{array}{l}
\frac{1}{{\cosh \left( {2v - 1} \right)t}} \times {\left( {\frac{1}{{\cosh t}}} \right)^{r - 1}},\,\,\,\,\,\,\,\,\,\,\left( {\frac{1}{2} \le v \le 1} \right),\\
\frac{1}{{\cosh \left( {1 - 2v} \right)t}} \times {\left( {\frac{1}{{\cosh t}}} \right)^{r - 1}},\,\,\,\,\,\,\,\,\,\,\left( {0 \le v \le \frac{1}{2}} \right).
\end{array} \right.\]
\end{itemize}
\end{remark}


It may be interesting to compare $K_r$ with the Lehmer mean:
$$
L_{\alpha}:=L_{\alpha}(a,b)=\dfrac{a^{\alpha}+b^{\alpha}}{a^{\alpha-1}+b^{\alpha-1}},\,\,\,{\rm }\,\,\,0\le \alpha \le 1.
$$
Although the Lehmer mean is often defined for $-\infty\le \alpha \le \infty$ in \cite[Chapter 8]{K2011} and \cite[Section 2.4]{BK2007}, we study it for $0\le \alpha \le 1$ since it is a symmetric homogeneous mean when
$0\le \alpha \le 1$ by $\dfrac{dL_{\alpha}(x,1)}{dx}=\dfrac{x^{\alpha}\left((1-\alpha)+\alpha x+x^{\alpha}\right)}{(x+x^{\alpha})^2}\ge 0$ for $0\le \alpha \le 1$ and $x>0$.
It is known that $L_0=H$, $L_{1/2}=G$ and $L_1=A$ so that the Lehmer mean $L_{\alpha}$ interpolates $H,G$ and $A$ as the mean $K_r$ so. Since $\dfrac{dL_{\alpha}(x,1)}{d\alpha}=\dfrac{x^{\alpha+1}(x-1)\log x}{(x^{\alpha}+x)^2}\geq 0$ for all $\alpha\in \mathbb{R}$ and 
$\dfrac{dK_r(x,1)}{dr}=\dfrac{x^{r/2}}{2}\left(\dfrac{x+1}{2}\right)^{1-r}\log\dfrac{4x}{(x+1)^2}\le 0$ for all $r\in \mathbb{R}$, we easily find that
$$
1\le r \le 2 \,\,\,{\rm and} \,\,\,  \frac{1}{2} \le \alpha \le 1 \Longrightarrow K_r \le L_{\alpha}
$$
and
$$
0\le r \le 1 \,\,\,{\rm and} \,\,\,0\le \alpha \le \frac{1}{2} \Longrightarrow K_r \ge L_{\alpha}.
$$
These are natural results since $K_2=L_0=H$, $K_1=L_{1/2}=G$ and $K_0=L_1=A$.
We give the stronger order between $K_r$ and $L_{\alpha}$ in the following.
\begin{proposition}\label{prop_Kr_Lehmer}
Let $\alpha,r \in \mathbb{R}$.
If $1\le r \le 2$ and $\frac{1}{2} \le \alpha \le 1$, then $K_r\preceq L_{\alpha}$.
If $0\le r \le 1$ and $0\le \alpha \le \frac{1}{2}$, then $L_{\alpha}\preceq K_r$.
\end{proposition}
\begin{proof}
We use the well known fact \cite[Eq.(1.5)]{HK1999India} that $\dfrac{\cosh ax}{\cosh bx}$ is positive definite if $b \ge a \ge 0$. We compute
$$
\dfrac{K_r(e^{2t},1)}{L_{\alpha}(e^{2t},1)}=\dfrac{\cosh(\alpha-1)t}{\cosh \alpha t}\times \left(\frac{1}{\cosh t}\right)^{r-1}=\dfrac{\cosh(1-\alpha)t}{\cosh \alpha t}\times\left(\frac{1}{\cosh t}\right)^{r-1}.
$$
Thus we have our first assertion. The second assertion follows similarly.
\end{proof}

We study here two bounds of the logarithmic mean $L$. That is, we consider the monotonicity of the parameters $r\in [0,1]$ and $s\in [0,1]$ for the mean $\vertiii{K_r(S,T)X}$ and the Heron mean $\vertiii{H_s(S,T)X}$, respectively.
It is easy to see $f(r):=\vert|K_r(S,T)X\vert|_2$   is monotone decrasing for $r\in [0,2]$.
Then we also have the following result for the unitarily invariant norm.
\begin{proposition}\label{prop_3.4K}
Let $r,r'\in [0,1]$. Then we have
$$
\vertiii{K_r(S,T)X}\le \vertiii{K_{r'}(S,T)X} \Longleftrightarrow r\ge r'.
$$
\end{proposition}
\begin{proof}
The inequality $\vertiii{K_r(S,T)X}\le \vertiii{K_{r'}(S,T)X}$ is equivalent to the positive definiteness of the following function
$$
\frac{K_r(e^{2t},1)}{K_{r'}(e^{2t},1)}=\left(\dfrac{1}{\cosh t}\right)^{r-r'}
$$
which is positive definite if and only if $r\ge r'$, since  the function $\dfrac{1}{\cosh x}$ is infinitely divisible.
\end{proof}
 Lemma \ref{le2.1} and  Proposition \ref{prop_2.4} clearly state that we have
$$
\vert|L(S,T)X\vert|_2\le \vert|H_s(S,T)X\vert|_2 \Longleftrightarrow 0\le s \le \frac{2}{3}.
$$
We also have that $f(s):=\vert|H_s(S,T)X\vert|_2$ is monotone decrasing for $s\in [0,1]$.

In addition, we have the following result for the unitarily invariant norm.
\begin{proposition}\label{prop_3.4}
Let $s,s'\in [0,1]$. Then we have
$$
\vertiii{H_s(S,T)X}\le \vertiii{H_{s'}(S,T)X} \Longleftrightarrow s\ge s'\quad {\rm and}\quad s'\le 1/2.
$$
\end{proposition}

\begin{proof}
The inequality 
$$\vertiii{H_s(S,T)X}\le \vertiii{H_{s'}(S,T)X}$$
 is equivalent to
 the positive definiteness of the following function
$$
\frac{H_s(e^{2t},1)}{H_{s'}(e^{2t},1)}
=\frac{1-s}{1-s'}\times \frac{\dfrac{s}{1-s}+\cosh t}{\dfrac{s'}{1-s'}+\cosh t},
$$
which is positive definite \cite[Proposition 7.3 (i)]{K2011} if and only if 
\begin{equation}\label{cond_prop3.4}
\dfrac{s}{1-s}\ge \dfrac{s'}{1-s'}\quad {\rm and} \quad \dfrac{s'}{1-s'}\le 1
\end{equation}
 since $\dfrac{s}{1-s}, \dfrac{s'}{1-s'} \ge 0$ from the assumption $s,s'\in [0,1]$. The conditions \eqref{cond_prop3.4} are equivalent to $s\ge s'$ and $s'\le 1/2$.
\end{proof}

\begin{theorem}\label{prop_opt}
Let $s\in [0,1]$. Then, 
$$
\vertiii{L(S,T)X} \le \vertiii{H_s(S,T)X} \Longleftrightarrow 0\le s \le \frac{1}{2}.
$$
\end{theorem}
\begin{proof}
Proposition \ref{prop_3.4} and the known inequality \cite[Example 3.5]{HK1999India}, (See also \cite[Corollary 4]{B2006} and \cite[Theorem 3.9]{KS2014}) $\vertiii{L(S,T)X} \le \vertiii{H_{1/2}(S,T)X}$  state that we have
$$
0\le s \le \frac{1}{2}  \Longrightarrow  \vertiii{L(S,T)X} \le \vertiii{H_s(S,T)X}.
$$

Conversely, we assume $\vertiii{L(S,T)X} \le \vertiii{H_s(S,T)X}$ which is equivalent to the positive definiteness of the following function 
$$
\frac{L(e^{2t},1)}{H_s(e^{2t},1)}=\frac{\sinh t}{t}\times \frac{1}{1-s}\times \frac{1}{\dfrac{s}{1-s}+\cosh t}=:\phi_1(t).
$$
Since $\phi_2(t):=\dfrac{t}{\sinh t}$ is positive definite, the product of the functions $\phi_1(t)\phi_2(t)=\dfrac{1}{\dfrac{s}{1-s}+\cosh t}$ has to be positive definite. Since the function $\dfrac{1}{\cosh x+\beta}$ is positive definite if and only if $-1<\beta \le 1$ \cite[Eq.(7.1)]{K2011}, we have $0\le s \le \dfrac{1}{2}$ under the assumption $s\in [0,1]$. Thus we obtain
$$
\vertiii{L(S,T)X} \le \vertiii{H_s(S,T)X} \Longrightarrow 0\le s \le \frac{1}{2}.
$$
\end{proof}

We state the other upper bound of the logarithmic mean. According to Lemma \ref{le2.1}, we have the relation $L\le B_{1/3}\le H_{2/3}$. We compare the means $L$ and $B_p$ in unitarily invariant norm. By the monotonicity of the binomial mean \cite[Theorem 9]{K2014} that $B_{p'}\preceq B_p$ for $p'\le p $, we easily see $G=B_0\preceq B_p$ for $p\geq 0$.
\begin{proposition}\label{prop_3.7}
If $p\ge \dfrac{1}{2}$, then $\vertiii{L(S,T)X}\le \vertiii{B_p(S,T)X}$. We also have
$$
p\le 0 \Longleftrightarrow \vertiii{L(S,T)X}\ge \vertiii{B_p(S,T)X}.
$$
\end{proposition}

\begin{proof}
We compute
$$
\frac{L(e^{2t},1)}{B_p(e^{2t},1)}=\frac{\sinh t}{t\left(\cosh pt\right)^{1/p}}=\frac{\sinh t/2 \cosh t/2}{t/2\left(\cosh pt\right)^{1/p}}.
$$
It is easy to see the above ratio is positive definite for $p=1/2$, since
the function $\dfrac{\sinh x}{x\left(\beta +\cosh x\right)},\,\,(-1<\beta \le 1)$ is infinitely divisible \cite[Remark 4 (iii)]{K2008}. Thus we have 
$$\vertiii{L(S,T)X}\le \vertiii{B_{1/2}(S,T)X}.$$

In addition, it is known the monotonicity of the binomial mean \cite[Theorem 9]{K2014} that  for $p'\le p$,
$$
\frac{B_{p'}(e^{2t},1)}{B_{p}(e^{2t},1)}=\frac{\left(\cosh p't\right)^{1/p'}}{\left(\cosh pt\right)^{1/p}}
$$
is infinitely divisible (namely positive definite) so that we have
$$
p'\le p \Longrightarrow \vertiii{B_{p'}(S,T)X}\le \vertiii{B_{p}(S,T)X}.
$$
Therefore we have $\vertiii{L(S,T)X}\le \vertiii{B_p(S,T)X}$ if $p\ge \dfrac{1}{2}$.

The function
$$
\dfrac{B_p(e^{2t},1)}{L(e^{2t},1)}=\dfrac{t}{\sinh t}\times\left(\dfrac{1}{\cosh pt}\right)^{-1/p}=:\eta_p(t)
$$ 
is positive definite if $p\le 0$. In addition,  for $p>0$ we have $\lim\limits_{t\to\infty}\eta_p(t)=\infty$ and $\lim\limits_{t\to 0}\eta_p(t)=1$. In general, a positive definite function $f$ satisfies the property $|f(t)|\le f(0)$, \cite[Eq.(5.6)]{B2007}. However the function $\eta_p(t)$ does not satisfy this condition so that the function $\eta_p(t)$ is not positive definite for $p>0$. Thus we have the second statement.
\end{proof}

\begin{remark}\label{remark_3.5}
\begin{itemize}
\item[(i)] From Proposition \ref{prop_3.7}, we are interested in the relation $L\preceq B_p$ for the case $0<p<\dfrac{1}{2}$. We have $L\npreceq B_p$ for $0\le p \le \dfrac{1}{3}$ by the fact $L\npreceq B_{1/3}$ which was stated before Proposition \ref{sec3_prop3.1} (\cite[Example 3.5]{HK1999India}) and the monotonicity of the parameter $p\in \mathbb{R}$ in the binomial mean $\vertiii{B_p(S,T)X}$, \cite[Corollary 11]{K2014}. In addition, one of the eigenvalues of the $7\times 7$ matrix  ${\left[ {\eta_{3/7}\left( {{t_i} - {t_j}} \right)} \right]_{i,j = 1,\cdots,7}}$ with $t_i:=i$ for $i=1,\cdots,7$, takes the negative value by numerical computations. To find a counter-example of $L\preceq B_p$ for the value close to $p<1/2$, we need to prepare the large size matrix and take much time to compute numerically its eigenvalues.

\item[(ii)] As we stated in (i) above, $L\npreceq B_{1/3}$ and also $L\npreceq H_{2/3}$ from Theorem \ref{prop_opt},
 while we have the inequalities $L\le B_{1/3}\le  H_{2/3}$. Here we show the relation $B_{1/3}\npreceq  H_{2/3}$.
We assume that the function 
$$
\dfrac{B_{p}(e^{2t},1)}{H_{2/3}(e^{2t},1)}= \frac{3\left(\cosh pt\right)^{1/p}}{2+\cosh t}
$$
is positive definite. By multiplying the positive definite function $\left(\dfrac{1}{\cosh pt}\right)^{1/p}$, $(p>0)$ to $\dfrac{B_{p}(e^{2t},1)}{H_{2/3}(e^{2t},1)}$, the function  $\dfrac{1}{2+\cosh t}$ has to be positive definite. This contradicts the fact that the function $\dfrac{1}{\beta+\cosh t}$ is positive definite if and only if $-1<\beta \le 1$ \cite[Eq.(7.1)]{K2011}. Taking $p=\dfrac{1}{3}$ leads to $B_{1/3}\npreceq  H_{2/3}$.
\item[(iii)]Since the function $\dfrac{\sinh t}{t\left(\beta+\cosh t\right)},\,\,(-1 <\beta \le 1)$ is infinitely divisible \cite[Remark 4 (iii)]{K2008}, the function for $0\le s \le 1$ 
$$\dfrac{L(e^{2t},1)}{H_s(e^{2t},1)}=\dfrac{1}{1-s}\times\dfrac{\sinh t}{t\left(\dfrac{s}{1-s}+\cosh t\right)}$$
 is infinitely divisible if and only if $0\le s \le \dfrac{1}{2}$.
\item[(iv)] Let $s\in [0,1]$. Since $\dfrac{G(e^{2t},1)}{H_s(e^{2t},1)}=\dfrac{1}{1-s}\times\dfrac{1}{\dfrac{s}{1-s}+\cosh t}$, we have
\begin{equation}\label{remark3.5_eq01}
\vertiii{G(S,T)X} \le \vertiii{H_s(S,T)X}  \Longleftrightarrow 0\le s \le \frac{1}{2}\,\,{\rm or}\,\,\, s=1.
\end{equation}
by the similar way to the proof of Theorem \ref{prop_opt}.
We assume that there exists a mean $M$ such that $G\preceq M\preceq  L$ and $M\neq G$. Then the condition on $s\in[0,1]$ satisfying $M\preceq H_s$ is just $0\le s \le \frac{1}{2}$ by Theorem \ref{prop_opt} and \eqref{remark3.5_eq01}.  A typical example for $M$ is the power difference mean $M_u$, $(1/2\le u \le 1)$.
\end{itemize}
\end{remark}

In \cite[Corollary 2]{D2006}, the unitarily invariant norm inequality between the Heinz mean and the logarithmic mean was shown as
$$
\vertiii{Hz_v(S,T)X}\le \vertiii{L(S,T)X}\Longleftrightarrow \frac{1}{4}\le v \le \frac{3}{4},
$$
where $Hz_v(S,T)X:=\dfrac{1}{2}\left(S^vXT^{1-v}+S^{1-v}XT^v\right)$.
It may be natural to compare the Heinz mean $\vertiii{Hz_v(S,T)X}$ and the Heron mean
$\vertiii{H_s(S,T)X}$. See \cite[Theorem 3.8]{KS2014} for this comparison.

\begin{proposition}\label{prop_3.6}
Let $s,v\in[0,1]$. Then,
$$
Hz_v\preceq H_s \Longleftrightarrow {\rm(c1), \,\,\, (c2)\,\,\, or \,\,\, (c3)},
$$
where
\begin{itemize}
\item[(c1)] $0\le s \le \dfrac{1}{2}\,\,$, $\dfrac{1}{2}\le v \le 1\,\,$ and $\,\,\,2v-1 \le \dfrac{\pi/2}{\pi-\cos^{-1}\left(\dfrac{s}{1-s}\right)}$,
\item[(c2)] $0\le s \le \dfrac{1}{2}\,\,$, $0\le v \le \dfrac{1}{2}\,\,$ and $\,\,\,1-2v \le \dfrac{\pi/2}{\pi-\cos^{-1}\left(\dfrac{s}{1-s}\right)}$.
\item[(c3)] $s=1$ and $v=\dfrac{1}{2}$.
\end{itemize}
\end{proposition}
\begin{proof}
We calculate
$$
\dfrac{Hz_v(e^t,1)}{H_s(e^t,1)}=\dfrac{e^{\left(v-1/2\right)t}+e^{\left(1/2-v\right)t}}{2\left(s+(1-s)\cosh t/2\right)}.
$$
For $s=1$, since the function $\cosh x,\,\,(x\neq 0)$ is not positive definite, it is necessary that $v=1/2$, which is (c3).
For $\dfrac{1}{2}\le v \le 1$, the above function is deformed as
$$
\dfrac{Hz_v(e^t,1)}{H_s(e^t,1)}=\frac{1}{1-s}\times \frac{\cosh\left((2v-1)t/2\right)}{\dfrac{s}{1-s}+\cosh t/2}.
$$
From \cite[Proposition 7.3 (ii)]{K2011}, the positive definiteness of the above function is equivalent to
$0\le s \le \dfrac{1}{2}$ and  $2v-1 \le \dfrac{\pi/2}{\pi-\cos^{-1}\left(\dfrac{s}{1-s}\right)}$ under the assumption $\dfrac{1}{2}\le v \le 1$. This yields to (c1). (c2) is similarly obtained. 
\end{proof}

We have the following results by taking $s=\dfrac{1}{2},\,\,\dfrac{1}{3},\,\,0$ in Proposition \ref{prop_3.6}, respectively.
\begin{eqnarray*}
&& \vertiii{Hz_v(S,T)X}\le \vertiii{H_{1/2}(S,T)X} \Longleftrightarrow \dfrac{1}{4}\le v \le \dfrac{3}{4}, \\
&& \vertiii{Hz_v(S,T)X}\le \vertiii{H_{1/3}(S,T)X} \Longleftrightarrow \dfrac{1}{8}\le v \le \dfrac{7}{8}, \\
&& \vertiii{Hz_v(S,T)X}\le \vertiii{A(S,T)X} \Longleftrightarrow 0 \le v \le 1.
\end{eqnarray*}


Related to Remark \ref{remark_3.5}, we consider the comparison of the binomial mean and the Heron mean. We note that $B_{1/2}(a,b)=H_{1/2}(a,b)$ and $B_{1/3}(a,b)\le H_{2/3}(a,b)$. We are interested in finding the condition on parameters $s\in [0,1]$ and $p\in\mathbb{R}$ such that $B_p\preceq H_s$ or $H_s\preceq B_p$.

\begin{proposition}\label{prop_3.9}
Let $p\in\mathbb{R}$ and $s\in [0,1]$. Then we have
$$
p\le \dfrac{1}{2}\,\, {\rm and}\,\,0\le s \le \dfrac{1}{2} \Longrightarrow \vertiii{B_p(S,T)X}\le \vertiii{H_s(S,T)X}
$$
and
$$
p\ge \dfrac{1}{2}\,\, {\rm and}\,\, \dfrac{1}{2}\le s \le 1 \Longrightarrow \vertiii{B_p(S,T)X}\ge \vertiii{H_s(S,T)X}.
$$
\end{proposition}
\begin{proof}
We compute
$$
\frac{B_p(e^{2t},1)}{H_s(e^{2t},1)}=\frac{B_p(e^{2t},1)}{B_{1/2}(e^{2t},1)} \times \frac{B_{1/2}(e^{2t},1)}{H_s(e^{2t},1)}.
$$
Since $\dfrac{B_p(e^{2t},1)}{B_{1/2}(e^{2t},1)}$ is positive definite when $p\le \dfrac{1}{2}$ by the monotonicity of the binomial mean \cite[Theorem 9]{K2014}, it is sufficient to prove the positive definiteness of 
$$
\frac{B_{1/2}(e^{2t},1)}{H_s(e^{2t},1)}=\frac{\left(\dfrac{e^t+1}{2}\right)^2}{se^t+(1-s)\left(\dfrac{e^{2t}+1}{2}\right)}=\frac{1}{2(1-s)}\times \frac{\cosh t +1}{\cosh t+\dfrac{s}{1-s}}.
$$
This function is positive definite if and only if $\dfrac{s}{1-s}\le 1 \Longleftrightarrow 0\le s \le\dfrac{1}{2}$.
The case $p\ge \dfrac{1}{2}$ and $\dfrac{1}{2}\le s \le 1$ can be similarly shown.
\end{proof}

As we stated in the end of Section \ref{sec2}, it is known \cite[Lemma 2.1]{DDF}  that we have the first inequalities below for $x>0$ and the second inequalities are trivial:
$$
0\le s \le \dfrac{1}{2}\Longrightarrow \left(\dfrac{x^s+1}{2}\right)^{1/s}\le (1-s)\sqrt{x}+s\left(\dfrac{x+1}{2}\right)\le s\sqrt{x}+(1-s)\left(\dfrac{x+1}{2}\right)
$$
and
$$
\dfrac{1}{2}\le s \le 1 \Longrightarrow \left(\dfrac{x^s+1}{2}\right)^{1/s}\ge (1-s)\sqrt{x}+s\left(\dfrac{x+1}{2}\right)\ge s\sqrt{x}+(1-s)\left(\dfrac{x+1}{2}\right).
$$
Taking $p=s$ in Proposition \ref{prop_3.9}, we see that the stronger relations $B_s\preceq H_s$ for $0\le s\le \dfrac{1}{2}$ and $H_s\preceq B_s$ for $\dfrac{1}{2}\le s \le 1$ than the scalar orderings $B_s \le H_s$ for $0\le s\le \dfrac{1}{2}$ and $H_s\le B_s$ for $\dfrac{1}{2}\le s \le 1$, respectively obtained.

In the comparison of the scalar orderings we have tighter relations  from \cite[Lemma 2.1]{DDF} such as $B_p\le \hat{H}_s$ for $0\le s\le \dfrac{1}{2}$ and $\hat{H}_s\le B_s$ for $\dfrac{1}{2}\le s \le 1$, respectively, where $\hat{H}_s:=\hat{H}_s(a,b)=(1-s)G+sA=H_{1-s}(a,b)$. We find from \cite[Theorem 7.1]{K2014} that $\hat{H}_s \preceq H_s$ for $0\le s \le \dfrac{1}{2}$ and $H_s\preceq \hat{H}_s$ for $\dfrac{1}{2}\le s \le 1$.
It is quite natural that we have an interest in the relation $B_s\preceq \hat{H}_s$ for $0\le s \le \dfrac{1}{2}$ and $\hat{H}_s\preceq B_s$ for $\dfrac{1}{2}\le s \le 1$.
However, these do not hold in general. We show the counter-examples. We compute the function
$\dfrac{B_s(e^{2t},1)}{\hat{H}_s(e^{2t},1)}=\dfrac{\left(\cosh st\right)^{1/s}}{(1-s)+s\cosh t}=:\psi_s(t)$.
 Then the eigenvalues
of the $3\times 3$ matrix ${\left[ {\psi_{1/4}\left( {{t_i} - {t_j}} \right)} \right]_{i,j = 1,2,3}}$ with $t_1:=1,\,\,t_2:=2,\,\,t_3:=3$ are approximately obtained as $2.96626, 0.0436155$ and $-0.00987773$ by the numerical computation. Thus the function $\psi_{1/4}(t)$  is not positive definite so that we have $B_{1/4}\npreceq \hat{H}_{1/4}$. Similarly we compute the function $\dfrac{\hat{H}_s(e^{2t},1)}{B_s(e^{2t},1)}=\dfrac{(1-s)+s\cosh t}{\left(\cosh st\right)^{1/s}}=:\xi_s(t)$. Then the eigenvalues
of the $3\times 3$ matrix ${\left[ {\xi_{3/4}\left( {{t_i} - {t_j}} \right)} \right]_{i,j = 1,2,3}}$ with $t_1:=1,\,\,t_2:=2,\,\,t_3:=3$ are approximately obtained as $2.98432, 0.0182053$ and $-0.00252532$ by the numerical computation. Thus the function $\xi_{3/4}(t)$  is not positive definite so that we have $\hat{H}_{3/4}\npreceq B_{3/4}$.  
\begin{remark}
From Proposition \ref{prop_3.9} and $\hat{H}_s=H_{1-s}$, we have
$$
p\le \dfrac{1}{2} \,\,\,{\rm and}\,\,\, \dfrac{1}{2}\le s \le 1 \Longrightarrow \vertiii{B_p(S,T)X} \le \vertiii{\hat{H}_s(S,T)X}
$$
and
$$
p\ge \dfrac{1}{2} \,\,\,{\rm and}\,\,\, 0\le s \le \dfrac{1}{2} \Longrightarrow \vertiii{B_p(S,T)X} \ge \vertiii{\hat{H}_s(S,T)X}.
$$
\end{remark}
\section{A new mean inequality chain}\label{sec4}
For any positive real numbers $a$ and $b$, we have the chain mean inequality (e.g., \cite[p.167]{Bullen2003}):
$$
m\leq \h\leq G\leq \l \leq A\leq M, 
$$
where $m=\min\{a,b\}$, $M=\max\{a,b\}$.
Now we consider the ratio $\rho=\sqrt{\dfrac{2A}{A+G}}$. It is obvious that $1\leq \rho\leq \sqrt{2}$. Then we have the new chain mean inequality, by applying Lemma \ref{le2.1}.
We also prepare the following inequality which is itself interesting. 

\begin{lemma}\label{le2.3}
For  $x>0$, we have 
\[
 \frac{\sqrt{2(x+1)}}{\sqrt{x}+1}\cdot \frac{x+1}{2}\leq \frac{x+1+|x-1|}{2}.
\]
\end{lemma}
\begin{proof}
The function $f(x)=x^{3/2}$ is convex on $(0,+\infty)$. So we can write
\[
f\left(\frac{x+1}{2}\right) \leq \frac{f(x)+f(1)}{2}
\]
which implies
\[
\sqrt{2(x+1)}\cdot \frac{x+1}{2} \leq x\sqrt{x}+1.
\]
Therefore we have
\[\frac{{\sqrt {2\left( {x + 1} \right)} }}{{\sqrt x  + 1}} \cdot \frac{{x + 1}}{2} \le \frac{{x\sqrt x  + 1}}{{\sqrt x  + 1}} = x - \sqrt x  + 1 \le \left\{ \begin{array}{l}
x\,\,\,\,\,\left( {\text{if}\,\,\,\,x \ge 1} \right)\\
1\,\,\,\,\,\,\left( {\text{if}\,\,\,\,0 < x \le 1} \right)
\end{array} \right.\]
which implies
\[
 \frac{\sqrt{2x+2}}{\sqrt{x}+1}\cdot \frac{x+1}{2}\leq \max \{1,x\}=\frac{x+1+|x-1|}{2}.
\]
\end{proof}

\begin{theorem}\label{th3.1}
For $a,b>0$ we have
$$
m\leq \rho\, m\leq H\leq \rho\, H\leq G\leq \rho\, G \leq L \leq \rho\, L \leq A \leq  \rho\, A\leq M.
$$
\end{theorem}

\begin{proof}
The first, third, fifth, seventh and ninth inequalities are trivial, since $\rho \ge 1$.
\begin{enumerate}
\item[(i)]
We get the tenth inequality by putting $x=\dfrac{a}{b}$ in Lemma \ref{le2.3}, since we have 
$$ \frac{\sqrt{2x+2}}{\sqrt{x}+1}\cdot \frac{x+1}{2}=\sqrt{\dfrac{x+1}{\dfrac{x+1}{2} +\sqrt{x}}}\cdot \frac{x+1}{2}.$$
\item[(ii)]
Since  we know that $L \leq \dfrac{2}{3}G+\dfrac{1}{3}A\le \dfrac{A+G}{2}$ from the third inequality of \eqref{fund_ineq}. Thus we have the eighth inequality:
\[
\rho \,L=\sqrt{\frac{2A}{A+G}} L=\sqrt{\frac{2 L^2A}{A+G}}\leq \sqrt{L A}\leq A
\]
\item[(iii)]
From Corollary \ref{co2.2} and the third inequality of \eqref{fund_ineq}, we have
$$
\rho\,G=\sqrt{\frac{2AG^2}{A+G}}\le \sqrt{\frac{2L^3}{A+G}}\le \sqrt{L^2}=L
$$
which gives the sixth inequality.
\item[(iv)]
Since we have
\[
\rho\,H=\sqrt{\frac{2H^2A}{A+G}}\leq \sqrt{\frac{H^2A}{G}}=\sqrt{HG}\leq G,
\]
we have the forth inequality.
\item[(v)]
The second inqaulity is proven by the similar argument as in (i).
\end{enumerate}
\end{proof}

\section*{Acknowledgment}
The authors would like to  express our deepest gratitude to Professor Hideki Kosaki for giving us valuable comments and suggestions to improve this manuscript.


%
\end{document}